\documentclass{amsart}
\usepackage{amssymb,amsmath, amsthm,latexsym}
\usepackage{graphics}
\usepackage{amscd}
\usepackage{graphics}
\usepackage{color}
\newcommand{\cal}[1]{\mathcal{#1}}
\theoremstyle{plain}
\newtheorem{theorem}{Theorem}

\newtheorem{lemma}{Lemma}[section]
\newtheorem{theo}[lemma]{Theorem}
\newtheorem{proposition}[lemma]{Proposition}
\newtheorem{corollary}[lemma]{Corollary}
\theoremstyle{definition}

\newtheorem{definition}[lemma]{Definition}
\newtheorem{example}[lemma]{Example}

\parskip=\bigskipamount

\let\egthree=\phi
\let\phi=\varphi
\let\varphi=\egthree

\newcounter{sebcomments}
\begin{document}
\title{Submanifold projections and hyperbolicity in ${\rm Out}(F_n)$}
\author{Ursula Hamenst\"adt and Sebastian Hensel}
\thanks
{AMS subject classification: 20M34\\
Partially supported by the
Hausdorff Center Bonn}
\date{March 21, 2024}

\begin{abstract}
The free splitting graph of a free group $F_n$ with 
$n\geq 2$ generators is a hyperbolic ${\rm Out}(F_n)$-graph which 
has a geometric realization as a sphere graph in the connected 
sum of $n$ copies of $S^1\times S^2$. We use this realization to construct
submanifold projections of the free splitting graph into the free splitting
graphs of proper free factors. This is used to construct for $n\geq 3$
a new hyperbolic ${\rm Out}(F_n)$-graph. If $n=3$, then 
every exponentially growing 
element acts on this graph with positive translation length.
\end{abstract}

\maketitle


\section{Introduction}


The \emph{free factor graph} ${\cal F\cal F}(F_n)$ 
for a free group $F_n$ of rank $n\geq 2$  is the graph whose
vertices are conjugacy classes of free factors of $F_n$ and where 
two such free factors $A_1,A_2$ are connected by an edge of length one if 
up to a global conjugation we have $A_1\subset  A_2$ or $A_2\subset A_1$.
The free factor graph is a locally infinite 
 Gromov hyperbolic geodesic metric graph, and the 
outer automorphism group ${\rm Out}(F_n)$ of $F_n$ acts as a group of 
simplicial automorphisms on ${\cal F\cal F}(F_n)$ \cite{BF14}.


There are other natural Gromov hyperbolic geodesic metric 
${\rm Out}(F_n)$-graphs. The best known is the 
so-called \emph{free splitting graph} 
\cite{HM13}, 
whose first barycentric subdivision ${\cal F\cal S}(F_n)$ 
is defined as follows. 
The vertices of  ${\cal F\cal S}(F_n)$ are
graph of groups decompositions of $F_n$ with 
trivial edge groups. Two such graph of groups
decompositions $G,G^\prime$ are connected
by an edge of length one if $G^\prime$ either
is a collapse or a blow-up of $G$.

In view of the geometric understanding of the mapping class
group of a closed surface $S$ of genus at least $2$ via its 
action on the curve graph of $S$ and the curve graph of 
subsurfaces using subsurface projections,
the graph ${\cal F\cal S}(F_n)$ 
is significant for the geometric 
understanding of ${\rm Out}(F_n)$. However, much less is known about 
${\cal F\cal S}(F_n)$ than about  
the free factor graph,
and the action of ${\rm Out}(F_n)$ 
is more complicated. 
For example, it was observed in 
\cite{HM19} that for sufficiently large $n$ 
there are free abelian subgroups of 
${\rm Out}(F_n)$ which act by loxodromic
isometries on ${\cal F\cal S}(F_n)$, with the same pair of fixed points on 
the Gromov boundary of ${\cal F\cal S}(F_n)$.

In spite of this difficulty, it turns out that there is hyperbolicity in 
${\rm Out}(F_n)$ beyond the free splitting graph. This is clear for 
$n=2$ since ${\rm Out}(F_2)={\rm GL}(2,\mathbb{Z})$ is a hyperbolic group. 
The following is our main result.

\begin{theorem}\label{main}
For $n\geq 3$ there exists a hyperbolic geodesic metric ${\rm Out}(F_n)$-graph 
${\cal P\cal G}_n$ which admits 
an equivariant one-Lipschitz projection onto the free splitting graph.
If $n=3$ then every exponentially growing 
automorphism acts with positive translation length on ${\cal P\cal G}_n$. 
\end{theorem}

Although for $n\geq 4$ the graph ${\cal P\cal G}_n$ does not have the property that
every exponentially growing automorphism acts on it with positive translation length,
we conjecture that such a hyperbolic ${\rm Out}(F_n)$-graph exists for all $n$.

Theorem \ref{main} can be thought of as a strengthening in rank $3$ 
of the 
following main result of \cite{BF14b}.

\begin{theorem}[Theorem 5.1 of \cite{BF14b}]\label{main1}
The group ${\rm Out}(F_n)$ acts by isometries on a product 
$Y=Y_1\times \cdots \times Y_{k}$ of $k>n$ hyperbolic spaces so that every exponentially growing
automorphism has positive translation length.  
\end{theorem}

While the proof of Theorem \ref{main1} uses the free factor graph and the 
action of
${\rm Out}(F_n)$ on Outer space as the main tool, 
we use 
a more topological viewpoint based on 
the 
so-called \emph{sphere system graph} \cite{HV96} which is defined 
as follows.

Let $M=S^1\times S^2\sharp \dots\sharp S^1\times S^2$ be the connected sum of 
$n$ copies of $S^1\times S^2$. Then $M$ is a closed manifold whose fundamental
group equals the free group $F_n$ with $n$ generators.

 A \emph{sphere} in $M$ is an embedded sphere which is 
not homotopic to 
zero. A \emph{sphere system} is a collection of pairwise
disjoint not mutually homotopic spheres in $M$. 
The sphere system is called \emph{simple} if
it decomposes $M$ into a union of balls.

Denote by ${\cal S\cal S\cal G}_n$ 
the locally finite 
graph whose vertices are isotopy classes of 
simple sphere systems in $M$ 
and where two such simple sphere  systems
are connected by an edge of length one if they can be realized disjointly. 
The group ${\rm Out}(F_n)$ acts on the graph ${\cal S\cal S\cal G}_n$
properly and cocompactly by work of Laudenbach \cite{Laudenbach}. Thus ${\cal S\cal S\cal G}_n$ is a geometric model for
${\rm Out}(F_n)$. 

Any sphere in $M$ defines up to conjugation a one-edge
free splitting of $F_n$, that is, a vertex in
${\cal F\cal S}(F_n)$, and 
two disjoint spheres $S_1,S_2$ define a two-edge free splitting which
collapses to the free splittings defined by $S_1,S_2$, that is, they
define an edge in ${\cal F\cal S}(F_n)$.
Thus the \emph{sphere graph} ${\cal S\cal G}_n$ whose set of vertices
is the set of isotopy classes of spheres in $M$ and whose edges connect spheres which
can be realized disjointly is a topological model for the free splitting graph. 
There also is a natural coarsely well defined coarsely ${\rm Out}(F_n)$-equivariant two-Lipschitz
projection
\[\Theta:{\cal S\cal S\cal G}_n\to {\cal S\cal G}_n\]
which associates to a simple sphere 
system one of its components.

As for Outer space, there are distinguished paths in
${\cal S\cal S\cal G}_n$ connecting 
any two simple spheres systems as follows. 
Let $S$ be a sphere which 
intersects the simple sphere system $\Sigma$. Assume that $S$ is in minimal
position with respect to $\Sigma$; this implies that $S$ intersects $\Sigma$ in the 
minimal number of components, and each of these components is an embedded circle in 
$S$  (see \cite{HiHo17} for a detailed account on these facts).

An innermost such  circle 
bounds an embedded disk $D$ in $S-\Sigma$. 
Its boundary $\partial D$ is contained in a sphere
$S_0\in \Sigma$. Replace $S_0$ by
the spheres obtained by
gluing $D$ to each of the two components of $S_0-D$.
These  spheres are disjoint from $\Sigma$. 
By Lemma 3.1 of \cite{HV96}, the sphere system 
$\Sigma_1$, obtained from the union of these two spheres with
$\Sigma-S_0$ 
by removing parallel copies of the same sphere if there are any, is simple,
and it has fewer intersections with $S$ than $\Sigma$. 
We call $\Sigma_1$
a \emph{sphere system obtained by surgery} of 
$\Sigma$ along $S$. Note that this notion is also defined if $S$ is a component of a sphere system
$\Sigma^\prime$. 

Repetition of this construction gives rise to so-called \emph{surgery sequences}
which are distinguished paths in ${\cal S\cal S\cal G}_n$. It was shown in 
\cite{HiHo17} that there exists a number $L>1$ such that 
the image by the map $\Theta$ 
of such a path is an \emph{unparameterized 
$L$-quasi-geodesic} in ${\cal S\cal G}_n$: 
there exists an increasing 
homeomorphism $\rho:[a,b]\to [0,m]$ such that the path 
$\Theta \circ \rho$ is an $L$-quasi-geodesic, that is, it satisfies
\[d_{\cal S\cal G}(\Theta \circ \rho(s),\Theta \circ \rho(t))/L-L
\leq \vert s-t\vert \leq  Ld_{\cal S\cal G}(\Theta \circ \rho(s),
\Theta \circ \rho(t))+L\]
where $d_{\cal S\cal G}$ denotes the distance in the sphere graph.

We use this fact to control \emph{submanifold projections} of 
the sphere graph into the sphere graphs of manifolds 
$M(\sigma)$, obtained by cutting $M$ open along a \emph{non-separating} sphere 
$\sigma$ and filling in the boundary by 
attachig a ball to each boundary component. These submanifold projections are
defined as follows.

Let $\sigma\subset M$ be a non--separating sphere.
The manifold $M(\sigma)$ is homeomorphic to 
the product of $n-1$ copies of $S^1\times S^2$. Given a non-separatring sphere $S\subset M$ distinct
from $\sigma$, we define the \emph{projection} $p_{M(\sigma)}(S)\subset M(\sigma)$ 
of $S$ into $M(\sigma)$ as follows. 
If $S\subset M-\sigma$ then put $p_{M(\sigma)}(S)=S\subset M(\sigma)$. 
This is well defined as since $S$ is non-separating, it is essential 
as a sphere in $M(\sigma)$. 
If $S$ intersects $\sigma$, then
choose an innermost disk $D\subset S$ with boundary on $\sigma$ and define 
$p_{M(\sigma)}(S)$ to be the sphere in $M(\sigma)$ which is the union of $D$ with one of the two components of 
$\sigma-D$. We observe in Section \ref{subfactor} that this is indeed an essential sphere in 
$M(\sigma)$. Furthermore, it determines a point in the sphere graph of
$M(\sigma)$ which coarsely does not depend on choices. This projection extends to separating
spheres in the same way, with the exception of separating spheres disjoint from $\sigma$
which are inessential as spheres in $M(\sigma)$.
We use this projection and its geometric properties as 
the main tool for the construction of the graph 
${\cal P\cal S}_n$. 

In \cite{BF14b}, a notion of subsurface projection of a free factor into the free splitting complex
of another free factor is defined. Although this projection should be closely related to 
ours, the precise relation between these two constructions is unclear. The article \cite{SS12}
contains yet another approach.

The outline of this article is as follows. In Section \ref{sec:graphs}, we define a family of 
${\rm Out}(F_n)$-graphs 
and show that they interpolate between the free factor graph and the free splitting graph. 
We also show that these graphs are all hyperbolic. 

In Section \ref{uniform} we introduce the concept of \emph{exponential growth} 
for surgery sequences in the 
simple sphere system graph. We show that surgery sequences of exponential growth are
quasi-geodesics. Furthermore, a surgery sequence which projects to a parameterized 
quasi-geodesic in the sphere graph has exponential growth. However, this is not necessary for 
exponential growth. 

In Section \ref{sec:spontaneous} we give a detailed analysis of the case $n=2$. We show that
in this case, exponential growth of a surgery sequence is equivalent to stating that its
projection to the sphere graph is a parameterized quasi-geodesic. For $n\geq 3$  we also construct
surgery sequences which do not define quasi-geodesics  in the sphere system graph.

Section \ref{subfactor} is devoted to the construction 
of submanifold projections. Most importantly, we 
show the bounded geodesic image property which is an essential tool towards
the proof of Theorem \ref{main}.
The proof of Theorem \ref{main} is 
contained in Section \ref{sec:products}.

\section{Graphs of free factors}\label{sec:graphs}

 In this section we introduce a family of graphs which interpolate between 
 the free factor graph and the free splitting graph. We assume that 
 $n\geq 3$ throughout.

\begin{definition}\label{maxfree}
For $m\leq n-2$, 
the \emph{level $m$  free factor graph} is the graph ${\cal F\cal F}_m(F_n)$
whose vertices are conjugacy classes of free factors of rank $n-1$, 
and where two such free factors $A_1,A_2$ are connected by an edge of 
length one if up to a global conjugation, $A_1\cap A_2$ contains a free factor of 
rank $m$.
\end{definition}

Clearly the graphs ${\cal F\cal F}_m(F_n)$ are geodesic ${\rm Out}(F_n)$-graphs. 
Furthermore, they all have the same set of vertices, and for each $m\geq 2$ the 
vertex inclusion defines an embedding ${\cal F\cal F}_m(F_n)\to 
{\cal F\cal F}_{m-1}(F_n)$. In other words, ${\cal F\cal F}_m(F_n)$ is obtained 
from ${\cal F\cal F}_{m-1}(F_n)$ by deleting some edges. The next proposition
justifies the terminology.

\begin{proposition}\label{freefactor}
The vertex inclusion defines a $2$-quasi-isometry
\[{\cal F\cal F}_1(F_n)\to {\cal F\cal F}(F_n).\]
\end{proposition}
\begin{proof}
Since every vertex of ${\cal F\cal F}(F_n)$ is of distance one to 
a rank $n-1$ free factor, the image of the vertex inclusion 
${\cal F\cal F}_1(F_n)\to {\cal F\cal F}(F_n)$ 
is coarsely dense
in ${\cal F\cal F}(F_n)$. Furthermore, by construction, 
any edge path $(A_i)_{0\leq i\leq k}\subset {\cal F\cal F}_1(F_n)$ of length $k$ 
induces (non-uniquely)  an edge path
in ${\cal F\cal F}(F_n)$ of length $2k$ with the same endpoints 
by replacing an edge 
$(A_i,A_{i+1})$ in ${\cal F\cal F}_1(F_n)$ 
by an edge path $(A_i,B_i,A_{i+1})$ in ${\cal F\cal F}(F_n)$ 
of length two, where $B_i$ is a free factor contained in the intersection
 $A_i\cap A_{i+1}$ which exists by the definition of ${\cal F\cal F}_1(F_n)$.

Thus it suffices to show the following.
Let $A,B$ be corank one free factors and let $(A_i)$ be a geodesic in 
the free factor graph ${\cal F\cal F}(F_n)$ 
connecting $A$ to $B$. Then there exists a path 
$(A_i^\prime)$ in ${\cal F\cal F}_1(F_n)$ connecting $A$ to $B$ whose length  
does not exceed the length of the path $(A_i)$.

To show that this is the case, note first that if $(A_j,A_{j+1},A_{j+2})\subset 
{\cal F\cal F}(F_n)$ is an edge path of length $2$ and if 
we have $A_j\subset A_{j+1}\subset A_{j+2}$, then 
$A_j,A_{j+2}$ are connected by an edge in ${\cal F\cal F}(F_n)$ 
and hence $(A_j,A_{j+1},A_{j+2})$ is not
a subarc of any geodesic in ${\cal F\cal F}(F_n)$. Thus we may assume 
that for all $i$, we have $A_{2i-1}\subset A_{2i}\supset A_{2i+1}$.

Then for each $i$, we may replace $A_{2i}$ by a corank 1 free factor $A_{2i}^\prime$ containing $A_{2i}$.
Since $A_{2i-1}\subset (A_{2i-2}\cap A_{2i})$ for all $i$, 
this then defines an edge path in ${\cal F\cal F}_1(F_n)$ of half the length and the 
same endpoints, which is what we wanted to show.  
\end{proof}

\begin{example}\label{n=3} 
If $n=3$ then there is only one graph ${\cal F\cal F}_1(F_3)$, and by 
Proposition \ref{freefactor}, it is $2$-quasi-isometric to the free factor graph. 
\end{example}

Our next goal is to relate the graph ${\cal F\cal F}_{n-2}(F_n)$ to the free
splitting graph. We use a topological version of this graph which 
was worked out carefully in \cite{AS11}. 

\begin{lemma}\label{freesplitsphere}
The sphere graph of $M$ is a topological realization of the free splitting graph
${\cal F\cal S}(F_n)$.
\end{lemma}
\begin{proof} (Sketch)
Each sphere $S\in {\cal S\cal G}_n$ determines a 
one-edge free splitting of $F_n$. Namely, if $S$ is non-separating, then 
for a choice of a basepoint $x\in M-S$, the subgroup of $\pi_1(M)$ of all homotopy classes
of loops which are disjoint from $S$ is a free factor of $F_n$ of rank $n-1$, and 
$S$ defines a one-vertex one-loop free splitting (an HNN-extension) of $F_n$.
If $S$ is separating, then $S$ defines a one-edge free splitting of $F_n$ by the Seifert
van Kampen theorem.  

Now let $S^\prime$ be a sphere which is disjoint from $S$. Then with the same argument,
$S\cup S^\prime$ defines a two edge free splitting which collapses to both 
the splitting defined by $S$ and $S^\prime$. Thus the sphere graph
maps $2$-quasi-isometrically 
into ${\cal F\cal S}(F_n)$, with one-dense image. We refer to \cite{AS11} for a complete proof.
\end{proof}

We need two technical properties of the sphere graph ${\cal S\cal G}_n$.
The first is the following simple

\begin{lemma}\label{nonseparating}
The subgraph of ${\cal S\cal G}_n$ of all non-separating spheres in $M$ is 
convex embedded in ${\cal S\cal G}_n$: any two non-separating spheres can be connected
by a geodesic in ${\cal S\cal G}_n$ consisting of non-separating spheres.
\end{lemma}
\begin{proof}
Let $A,B$ be non-separating spheres and connect $A$ to $B$ by a geodesic 
$(S_j)_{0\leq j\leq m}$. For each $i$ consider the sphere $S_{2i+1}$. It is disjoint from
both $S_{2i}$ and $S_{2i+2}$. As $(S_j)$ is a geodesic, if $S_{2i+1}$ is separating then
$S_{2i},S_{2i+2}$ are contained in the same component $U$ of $M-S_{2i+1}$ since otherwise the sphere 
$S_{2i+1}$ can be deleted from the sequence. Choose a non-separating sphere $S_{2i+1}^\prime$
in the component
$M-U$ and replace $S_{2i+1}$ by $S_{2i+1}^\prime$. The resulting path is a geodesic, and each of the spheres
with odd index are non-separating, while the spheres with even index are unchanged.
Proceed in the same way with the spheres $S_{2i}$. 
\end{proof}

Define a subgraph ${\cal N\cal S\cal G}_n$ of ${\cal S\cal G}_n$ as follows.
The vertices of ${\cal N\cal S\cal G}_n$ are non-separating spheres, and two 
such spheres $S_1,S_2$ are connected by an edge of length one if they can be
realized disjointly and if moreover $M-(S_1\cup S_2)$ is connected. 

The following is the analog of a well-known result for curve graphs.
\begin{proposition}\label{quasiisometric}
The inclusion ${\cal N\cal S\cal G}_n\to {\cal S\cal G}_n$ is a $2$-quasi-isometry.
\end{proposition}
\begin{proof}
Since every separating sphere is of distance one
to a non-separating sphere,
the graph ${\cal N\cal S\cal G}_n$ is 
one-dense in ${\cal S\cal G}_n$. Furthermore, by Lemma \ref{nonseparating}, two 
vertices of ${\cal N\cal S\cal G}_n$ can be connected by a geodesic $(S_i)\subset  
{\cal S\cal G}_n$ consisting of non-separating spheres.

It is possible that in the path $(S_i)$, there are two adjacent spheres, say the 
spheres $S_i,S_{i+1}$,  which form a bounding pair,
that is, such that $M-(S_i\cup S_{i+1})$ is disconnected. We now replace successively each such pair
$S_i,S_{i+1}$ by an edge path $S_i,D_i,S_{i+1}$ of length two such that $M-(S_i\cup D_i)$ and 
$M-(D_i\cup S_{i+1})$ are both connected. To see that this is possible note that if a bounding pair
exists then $n\geq 3$. Then $M-(S_i\cup S_{i+1})$ contains a component which is a non-trivial connected 
sum of $S^1\times S^2$ with the interiors of two balls removed. Such a manifold contains a non-separating 
embedded sphere $D_i$. This sphere is disjoint from $S_i\cup S_{i+1}$, 
and $M-(S_i\cup D_i)$ and $M-(D_i\cup S_{i+1})$ are both connected. 

The length of the modified path $(S_i^\prime)$ is at most twice the length of the path $(S_i)$ connecting
the same endpoints. Furthermore, 
any two consecutive vertices $S_i^\prime,S_{i+1}^\prime$ of this path have the property that 
$M-(S_i^\prime \cup S_{i+1}^\prime)$ is connected. This completes the proof of the lemma.
\end{proof}

\begin{example}\label{fareygraph}
The free group $F_2$ with two generators is the fundamental group of a once punctured
torus $T$.  Each oriented non-peripheral simple closed curve $c$ on $T$ determines the conjugacy
class of 
a primitive element of $F_2$, and any conjugacy class of a primitive element arises in this
way. Now primitive elements in $F_2$ are precisely the generators of corank one free factors of 
$F_2$. Moreover, conjugacy classes of corank one free factors of $F_2$ are in bijection with non-separating
spheres in the manifold $M$. Thus
the vertices of ${\cal N\cal S\cal G}_2$ correspond precisely to the simple closed
curves on $T$.

Two such conjugacy classes are connected by an edge in ${\cal N\cal S\cal G}_2$ 
if they correspond to disjoint spheres in 
$M$. This is the case if and only if they define a free basis of $F_2$, which is the case if and only 
if the simple closed curves on $T$ defining these conjugacy classes
intersect up to homotopy in precisely one point. As a consequence, the graph ${\cal N\cal S\cal G}_2$ 
is nothing else than the familiar \emph{Farey graph}. 
\end{example}

The relation between the free splitting graph ${\cal F\cal S}(F_n)$ and the graph 
${\cal F\cal F}_{n-2}(F_n)$ is now a consequence of the following observation.

\begin{lemma}\label{spherefactor}
There exists a one-Lipschitz simplicial map ${\cal N\cal S\cal G}_n\to 
{\cal F\cal F}_{n-2}(F_n)$ which is surjective on vertices.
\end{lemma} 
\begin{proof} If $S_1,S_2$ are vertices in ${\cal N\cal S\cal G}_n$ which are connected
by an edge, then for a choice of a basepoint $x\in M-(S_1\cup S_2)$, 
the spheres $S_i$ define corank one free factors $A_1,A_2$ of $F_n=\pi_1(M,x)$ 
of homotopy classes of loops disjoint from $S_1,S_2$, and these free factors intersect 
in the corank $2$ free factor of homotopy classes of loops disjoint from both 
$S_1\cup S_2$. Thus the edge between $S_1$ and $S_2$ in ${\cal N\cal S\cal G}_n$ defines
an edge in the graph ${\cal F\cal F}_{n-2}(F_n)$ as claimed in the lemma.
\end{proof}

As an immediate consequence of Lemma \ref{freesplitsphere}, 
Lemma \ref{quasiisometric} and Lemma \ref{spherefactor}, we obtain 

\begin{corollary}\label{retract}
  There exists a coarse two-Lipschitz map
\[{\cal F\cal S}(F_n)\to 
{\cal F\cal F}_{n-2}(F_n)\] which is surjective on vertices.
\end{corollary}

\begin{example}\label{ex:rank3}
If $n=3$ then Proposition \ref{freefactor} shows that 
the free factor graph is 2-quasi-isometric to  the graph ${\cal F\cal F}_{n-2}(F_n)$.
However, it is very different from the free splitting graph. 
Indeed, there are elements of ${\rm Out}(F_3)$ which act on the 
free splitting graph as loxodromic isometries, but which fix a free factor.
Such an example is discussed in Example 4.2 of \cite{HM19}.
It can be constructed with the help of a 
\emph{relative train track map}.

The example can be viewed as a family of spheres in $M$ which are all disjoint from 
a fixed simple loop defining a generator of $F_3$, but contain tubes 
winding around the loop. 
\end{example}

Recall from the introduction
that a simple sphere system $\Sigma$ can be modified to another simple sphere 
system by a \emph{surgery move} in direction of a sphere system $\Sigma^\prime$ 
as follows. 
Let $S^\prime \in \Sigma^\prime$, assumed to be in minimal
position with respect to $\Sigma$. Then each component
of $S^\prime \cap \Sigma$  is an embedded circle in 
$S^\prime$. 

An innermost such circle bounds an embedded disk $D$ in $S^\prime$. Its boundary $\partial D$ 
is contained in a sphere $S\in \Sigma$. 
The two spheres obtained by
gluing $D$ to each of the two components of $S-\partial D$ are disjoint and 
disjoint from $\Sigma$. Let $\Sigma_1$ be 
the union of $\Sigma-S$ with these two spheres, with parallel
copies of the same sphere removed. 
By Lemma 3.1 of \cite{HV96}, the sphere system 
$\Sigma_1$ is simple, and it 
has fewer intersections with $\Sigma^\prime$ than $\Sigma$. 

Repetition of this construction, keeping the direction $\Sigma^\prime$ fixed (and starting
in a second step from $\Sigma_1$) 
are called \emph{surgery sequences}. 

Note that there is a natural coarsely well defined projection 
$\tau:{\cal N\cal S\cal G}_n\to {\cal F\cal F}_{m}(F_n)$
which factors through the composition of the map from Lemma \ref{spherefactor} 
with the inclusion ${\cal F\cal F}_{n-2}(F_n)\to {\cal F\cal F}_m(F_n)$. 
As in \cite{HiHo17}, we use the images of surgery sequences under the map 
$\tau$ and an argument of \cite{KR14} to show

\begin{theo}\label{conj1}
Each of the graphs ${\cal F\cal F}_m(F_n)$ $(m\leq n-2)$ is 
hyperbolic, and the natural projections of surgery paths are 
uniform unparameterized
quasi-geodesics in ${\cal F\cal F}_m(F_n)$.
\end{theo}
\begin{proof}
We follow \cite{HiHo17} (the proof of Theorem 8.3). Let $S_0,S_1$ be non-separating 
spheres and 
assume that $\tau(S_0)$ and $\tau(S_1)$ are connected by 
an edge in ${\cal F\cal F}_{m}(F_n)$. Then we can find 
an embedded rose $R$ in $M$ with vertex $p$ and with $m$ petals
so that the inclusion $\pi_1(R,p)\to \pi_1(M,p)$ is $\pi_1$-injective and such that both 
$S_0$ and $S_1$ are disjoint from $R$.

Namely, let $\tilde M$ be obtained from
$M$ by removing the interior of a small ball from $M$. Put a basepoint $p$ on the 
boundary of $\tilde M$. For any non-separating sphere $S$ in $M$ choose a lift 
$\tilde S$ of $S$ to $\tilde M$. If $\tau(S_0),\tau(S_1)$ are connected in 
${\cal F\cal F}_{m}(F_n)$ by an edge then there exists 
$g\in F_n$ such that $\pi_1(\tilde M-\tilde S_0,p)$ and 
$\pi_1(\tilde M-g\tilde S_1g^{-1},p)$ contain a free factor of rank $m$ defining the 
conjugacy class of a free factor as in the definition of an edge in ${\cal F\cal F}_{m}(F_n)$.
Here $g\tilde S_1g^{-1}$ denotes the image of $\tilde S_1$ under a diffeomorphism of 
$\tilde M$ realizing the conjugation by $g$. 

It follows from Lemma 2.2 of \cite{HV98} that this free factor can be represented as the 
fundamental group of a rose $R$ with $m$ petals and basepoint at $p$ which is disjoint 
from both $\tilde S_0$ and $g\tilde S_1\tilde g^{-1}$. Projection of this rose as well as the spheres
$\tilde S_0,\tilde S_1$ to $M$ yields the statement claimed in the first paragraph of this proof.

Since neither $S_0$ nor $S_1$ intersect the rose $R$, 
any surgery path connecting $S_0$ to $S_1$ consists of spheres disjoint from $R$. As surgery
paths are uniform unparameterized quasi-geodesics in ${\cal S\cal G}_n$ \cite{HiHo17} and 
hence give rise to uniform unparameterized
quasi-geodesics in ${\cal N\cal S\cal G}_n$ by Proposition \ref{quasiisometric}, this implies that the 
fibers of the projection $\tau$ are uniformly quasi-convex: Any two points in a fiber are connected
by a uniform quasi-geodesic in ${\cal N\cal S\cal G}_n$ 
which is entirely contained in this fiber.

As a consequence, we can apply 
the main result of \cite{KR14}. We conclude that indeed, for any $m\leq n-2$ 
the level $m$ free factor graph is hyperbolic, and surgery paths in ${\cal S\cal G}_n$ 
(that is, edge  paths in ${\cal N\cal S\cal G}_n$ at distance two from surgery paths in 
${\cal S\cal G}_n$) project to uniform
unparameterized quasi-geodesics in ${\cal F\cal F}_m(F_n)$.
\end{proof}

\section{Exponential growth}\label{uniform}

For any sphere system $\Sigma$ and any embbeded finite graph $R$ in $M=\sharp_n S^1\times S^2$
let \[\iota(\Sigma,R)\] be the minimal number of intersection points
between $\Sigma$ and a homotopic realization of $R$, 
counted with multiplicity. Equivalently, this is the minimal
number of intersection points between $\Sigma$ and 
a homotopic realization of $R$ such that 
every vertex of $R$ is contained in $M-\Sigma$.

A simple sphere system $\Sigma$ is \emph{reduced} 
if its complement is connected. Each reduced sphere
system is \emph{dual} to a unique isotopy class of
a rose $R\subset M$ which
defines the conjugacy class of a free basis for $F_n$. Here 
duality means that up to homotopy, 
each component of $\Sigma$ intersects 
the rose $R$ in a single point.

Recall that ${\rm Out}(F_n)$ can be generated by \emph{Nielsen moves}. Such a Nielsen
move either is a Nielsen twist or a permutation of two rank one free factors
in a free basis (up to conjugation).
A Nielsen twist
replaces a marked rose $R$ by another marked rose $R^\prime$.
There is a homotopy equivalence
$R^\prime\to R$ which maps a leaf of $R^\prime$ to a loop
in $R$  
which either is a single leaf of $R$ or passes through
precisely two leaves. Thus we have

\begin{lemma}\label{nielsentwist}
  Let $\Sigma$ be a simple
  sphere system, let $R$ be a marked rose and assume that
  $R^\prime$ is obtained from $R$ by a single Nielsen twist; then
  \[\iota(\Sigma,R^\prime)\in \left[\frac{1}{2} \iota(\Sigma,R),
 2 \iota(\Sigma,R)\right].\]   
\end{lemma}  
\begin{proof}
As the marked homotopy equivalence
$R^\prime\to R$ can be represented by a $2:1$ map, we have
$\iota(\Sigma,R^\prime)\leq 2 \iota(\Sigma,R)$.
On the other hand, the marked rose $R^\prime$ is obtained from $R$ by a single
Nielsen twist as well, which immediately shows the second part of the
inequality.
\end{proof}

Let ${\cal R}$ be the 
graph whose set of vertices is the set of all marked roses
and where two such roses are connected by an 
edge of length one if they are related by a Nielsen move. 
Then ${\cal R}$ is an ${\rm Out}(F_n)$-graph on which 
${\rm Out}(F_n)$ acts properly and cocompactly. In other words,
${\cal R}$ is a geometric realization of ${\rm Out}(F_n)$.

Sphere systems define another geometric realization of
${\rm Out}(F_n)$. Namely, let ${\cal S\cal S\cal G}_n$ be the simple
sphere system graph and let $d_{\cal S\cal S\cal G}$ be the distance
in ${\cal S\cal S\cal G}_n$. By invariance under the action of
${\rm Out}(F_n)$, cocompactness, and the fact that stabilisers of
simple sphere systems are finite, the sphere system graph is
equivariantly quasi-isometric to ${\rm Out}(F_n)$. 

Given any simple sphere system $\Sigma$, we can obtain a reduced 
sphere system by removal of some of the spheres. Such a
reduced sphere system admits a dual rose. We call a rose $R$ obtained
in this way \emph{dual} to $\Sigma$ although $R$ may not be unique.
The coarsely
well defined map ${\cal S\cal S\cal G}_n\to {\cal R}$
which associates to a simple sphere system a dual rose dual 
is a coaresely ${\rm Out}(F_n)$-equivariant quasi-isometry.

\begin{lemma}\label{intersectwo}
There exists a number $C_0>0$ with the following properties.
Let $\Sigma_0,\Sigma_1$ be reduced
sphere systems and let $R$ be a rose dual
to $\Sigma_1$; then
$d_{\cal S\cal S\cal G}(\Sigma_0,\Sigma_1)\geq C_0\log_2 \iota(\Sigma_0,R)$.
\end{lemma}
\begin{proof} In Lemma \ref{nielsentwist} we observed 
that each Nielsen move decreases intersection numbers between
a rose and a sphere system by at most a factor of two. Since the 
graph ${\cal S\cal S\cal G}_n$ is coarsely ${\rm Out}(F_n)$-equivariantly 
quasi-isometric to ${\cal R}$,  from this the lemma follows.
\end{proof}

Let $(\Sigma_i)_{0\leq i\leq m}$ be a surgery sequence of simple
sphere systems.
For each $i$ let $R_i$ be a rose dual to $\Sigma_i$.
Then $R_i$ defines a vertex in the graph 
${\cal R}$. The distance in ${\cal R}$ between
$R_i$ and $R_{i+1}$ is bounded from above independently of $i$.
We use this to observe

\begin{lemma}\label{lowerbound}
There exists $C_1 >0$ with the following property.
Let $(\Sigma_i)_{0\leq i\leq m}$ be a surgery sequence of simple
  sphere systems, and let $(R_i)$ be a sequence of dual roses; then
\[\iota(\Sigma_m,R_1)\geq C_1 \iota (\Sigma_m,R_0).\]  
\end{lemma}  
\begin{proof}
  By definition, the sphere system $\Sigma_1$ is obtained from $\Sigma_0$
  by one surgery operation, followed by removal of at most two spheres from the
  resulting system. Thus the dual rose $R_1$ is obtained from
  the rose $R_0$ by a uniformly bounded number of Nielsen twist (which are the
  generators of ${\rm Out}(F_n)$, permutations play no role here), say at most
  $\ell$ of such twists. The lemma now follows from
  Lemma \ref{nielsentwist}.
\end{proof}

Lemma \ref{lowerbound} shows that for a surgery sequence $(\Sigma_i)_{0\leq i\leq m}$,
intersection numbers with $\Sigma_m$ of roses
dual to $\Sigma_i$ 
decrease at most exponentially along the sequence, with a fixed exponent
depending only on $n$. We next look at such sequences for which
intersection numbers decrease uniformly exponentially.

\begin{definition}\label{def:exponentialgrowth}
For $a\in (0,1)$ and $k\geq 1$ 
the surgery sequence $(\Sigma_i)_{0\leq i\leq m}$ 
has \emph{$(a,k)$-exponential growth} if 
$\iota(\Sigma_m,R_{i+k})\leq a\iota(\Sigma_m,R_{i})$ 
and $\iota(\Sigma_0,R_i)\leq a\iota(\Sigma_0,R_{i+k})$ for all $i$.
\end{definition}

The next result shows that exponential growth yields geometric 
control.

\begin{theo}\label{quasigeo2}
For all $a\in (0,1),k\geq 1$
there is a number $\ell(a,k)>1$ with the following property.
Let $\Sigma,\Lambda$ be two simple sphere systems which  
are connected by a surgery sequence $(\Sigma_i)$.
Assume that this sequence has $(a,k)$-exponential growth.
For each $i$ let $R_i$ be a rose dual to $\Sigma_i$.
Then the sequence $(R_i)$ defines an $\ell(a,k)$-quasi-geodesic
in the graph ${\cal R}$.
\end{theo}
\begin{proof} For a number $L>1$, an \emph{$L$-Lipschitz retraction} of the 
graph ${\cal R}$ onto a subset $A\subset {\cal R}$ is an 
$L$-Lipschitz map $\Upsilon:{\cal R}\to A$ such that
$d(x,\Upsilon(x))\leq L$ for all $x\in A$. 
If there exists an $L$-Lipschitz retraction ${\cal R}\to A$ then since ${\cal R}$ 
is a geodesic metric graph, 
the inclusion $A\to {\cal R}$ is \emph{weakly $L$-quasi-convex:}
For any two points $x,y\in A$ there exists a path in the $L$-neighborhood of 
$A$ with the same endpoints which is an $L$-quasi-geodesic in 
${\cal R}$ (with additive constant larger than $L$).

As a consequence, it suffices to show that there is an  
$L$-Lipschitz
retraction of ${\cal R}$ onto a sequence $(R_i)_{0\leq i\leq m}$
of roses dual to the sphere systems $\Sigma_i$ for a constant $L>1$ only 
depending on $a,k$ (and, of course, the rank $n$).

Let $G\in {\cal R}$ be a marked rose. We assume that $G$ is 
embedded in $M$.
Let $\kappa=\log\frac{\iota(\Sigma,G)}{\iota(\Lambda,G)}$.
We say that $P(G)=R_i$ is 
\emph{roughly balanced} for $G$ if 
\[\log \frac{\iota(\Sigma,R_i)}{\iota(\Lambda,R_i)}
  \in [\kappa +\log C_1,\kappa- \log C_1 ]\]
where $C_1\in (0,1)$ is as in Lemma \ref{lowerbound}.
If $\kappa<\log\frac{\iota(\Sigma,R_0)}{\iota(\Lambda,R_0)}$ then we put
$P(G)=\Sigma$, and similarly we put $P(G)=\Lambda$ if 
$\kappa>\log \frac{\iota(\Sigma,R_m)}{\iota(\Lambda,R_m)}$. 
By Lemma \ref{lowerbound} and the choice of the constant
$C_1>0$, such a number $i$ exists, and  
Definition \ref{def:exponentialgrowth} yields that
it is coarsely unique:  
If $R_j$ is another such point then
$\vert j-i\vert \leq k \log\kappa/\log a$.

Now let us assume that 
$G^\prime$ is obtained from $G$ by a single
Nielsen twist. Then Lemma \ref{nielsentwist} shows that
\[\left\vert \log \frac{\iota (\Sigma,G^\prime)}{\iota(\Lambda,G^\prime)}-
\log \frac{\iota(\Sigma,G)}{\iota(\Lambda,G)} \right\vert \leq 2 \log 2.\]
Thus as a consequence of $(a,k)$-exponential growth, 
we obtain that 
the intrinsic distance between $P(G)$ and 
$P(G^\prime)$ is at most $L$ where 
$L=L(a,k)>0$ is a universal constant. 
In other words, the map $P$ is coarsely $L$-Lipschitz.

Now if $G$ is dual to one of the sphere
systems $\Sigma_i$ then it follows from 
the construction that $P(G)$ is contained
in a uniformly bounded neighborhood of 
$\Sigma_i$. As a consequence, $P$ is indeed a Lipschitz retraction. 
The lemma follows.
\end{proof}



While Lemma \ref{nielsentwist} shows that intersection numbers
change at most exponentially with a fixed rate along
a one-Lipschitz path in the graph ${\cal R}$,
the next observation yields that the distance in ${\cal S\cal G}_n$
yields a lower bound on intersection numbers.

\begin{lemma}\label{intersect}
Let $S\subset M$ be a sphere and
let $R\subset M$ be an embedded rose with
$n$ petals and vertex $p$ such that the inclusion
$R\to M$ defines an isomorphism of $\pi_1(R,p)\to \pi_1(M,p)$.  
Let $S^\prime\subset M$ be a sphere which intersects $R$ in 
precisely one point; 
then 
\[d_{\cal S\cal G}(S,S^\prime)\leq 2\log_2 \iota(S,R)+3.\]
\end{lemma}
\begin{proof} If $S^\prime$ is disjoint from $S$ then 
\[d_{\cal S\cal G}(S,S^\prime)=1\] and 
there is nothing
to show. Thus assume that
$S^\prime,S$ intersect and that 
$R$ intersects $S$ in $k\geq 1$ points. 

There are at least two innermost components of $S^\prime-S$. 
Up to homotopy of $R$, we may assume that  the intersection 
point between $R$ and $S^\prime$ is contained in one of these components, 
say the component $D^\prime$. 
Let $D$ be an innermost component of $S^\prime-S$ different from
$D^\prime$; its
boundary $\partial D$ decomposes $S$ into two disks 
$D_1,D_2$. Assume by renaming that the disk $D_1$ has fewer intersections
with $R$ than $D_2$. Then $R$ intersects $D_1$ in at most 
$k/2$ points.

Surger $S$ at $D$ so that the surgered sphere $S_1$ is the union
$D_1\cup D$. Then $\iota(S_1,R)\leq k/2$.
Note that $d_{{\cal S\cal G}}(S,S_1)\leq 1$ since 
$S,S_1$ are disjoint. 
The lemma now follows by induction on the length of a surgery 
sequence connecting $S$ to a sphere disjoint from $S^\prime$.
\end{proof}

Recall the coarsely well defined map $\Theta$
which associates to a  simple sphere system one of its components. 
For a number $B>1$, define two reduced sphere systems
$\Sigma_0,\Sigma_1$ to be  
in \emph{$B$-tight position} if 
$Bd_{\cal S\cal G}(\Theta(\Sigma_0),\Theta(\Sigma_1))\geq d_{\cal S\cal S\cal G}(\Sigma_0,\Sigma_1)$.

\begin{corollary}\label{intersectdetect}
For every $B>1$ there is a number $a=a(B)>0$ with the following 
property.
Let $\Sigma_0,\Sigma_1$ be two reduced 
sphere systems which are in $B$-tight position. 
Let $R_1$ be a rose dual to 
$S_1$; then 
\[d_{\cal S\cal S\cal G}(\Sigma_0,\Sigma_1)\in [\log_2 \iota(\Sigma_0,R_1)/a,
a\log_2\iota(\Sigma_0,R_1)].\] 
\end{corollary}
\begin{proof} Since $\Sigma_0,\Sigma_1$ are in $B$-tight position, 
we have 
\[d_{\cal S\cal S\cal G}(\Sigma_0,\Sigma_1)\leq 
Bd_{\cal S\cal G}(\Theta(\Sigma_0),\Theta(\Sigma_1)).\]
Thus the corollary follows from Lemma \ref{intersect} and
Lemma \ref{intersectwo}.
\end{proof}

\section{Growth and quasigeodesics}\label{sec:spontaneous}

The goal of this section is to give some additional geometric
information on surgery sequences in relation to growth. 
We begin with a detailed analysis of the case of rank $2$.

%
%
In this section we only consider particular surgery sequences called
\emph{full surgery sequences}, 
defined by the
property that we always use all spheres (and remove multiple
copies). That is, we replace a sphere by both spheres obtained from
surgery at a fixed innermost disk.

Recall the map $\Theta:{\cal S \cal S\cal G}_n\to {\cal S\cal G}_n$ which 
associates to a simple sphere system one of its components.
In the statement of the following 
proposition, exponential growth means 
$(a,k)$-exponential growth for some $a>0,k>0$. The constants 
depend on each other, but we do not make this dependence quantitative.

\begin{proposition}\label{casen=2}
For the free group of rank $n=2$ 
and a full surgery sequence $\Sigma_i\subset {\cal S\cal S\cal G}_2$ of simple
sphere systems the following are equivalent.
\begin{enumerate}
\item $\Sigma_i$ is of exponential growth.
\item The image sequence $\Theta(\Sigma_i)$ is a parameterized
quasi-geodesic in ${\cal S\cal G}_2$. 
\end{enumerate}
\end{proposition}
\begin{proof} Since  
(2) implies (1) by Lemma \ref{intersect} (and in fact, this implication holds
true for any $n\geq 2$), it suffices to show
that (1) implies (2).

We observed in Example \ref{fareygraph} that up to 
uniform quasi-isometry, the graph ${\cal S\cal G}_2$ can be identified
with the Farey graph, where this identification is via viewing the free group 
$F_2$ as the fundamental group of a once punctured torus $T$ and 
viewing the Farey graph as the curve graph of $T$. 

Furthermore, we have ${\rm Out}(F_2)={\rm GL}(2,\mathbb{Z})$, which is a hyperbolic
group with respect to some (and hence any) finite symmetric generating set. 
Thus any uniform (that is, with fixed constants) 
quasi-geodesic $\gamma$ in ${\rm Out}(F_2)$ is \emph{stable}:
Any other uniform quasi-geodesic with the same endpoints is contained 
in a uniformly bounded neighborhood of $\gamma$. Since the surgery sequence 
$\Sigma_i$ is of $(a,k)$-exponential growth by assumption, 
Theorem \ref{quasigeo2} shows that it determines a quasi-geodesic in ${\rm GL}(2,\mathbb{Z})$
and hence it is at uniformly bounded distance from a geodesic.

To understand the relation between the geometry of ${\rm Out}(F_2)$ and the geometry of 
the Farey graph we first pass to the quotient ${\rm PSL}(2,\mathbb{Z})$ of the index two subgroup
${\rm SL}(2,\mathbb{Z})$ of $GL(2,\mathbb{Z})$, with fiber of order $2$. 
It acts as a group of isometries on the hyperbolic 
plane $\mathbb{H}^2$.  The quotient of $\mathbb{H}^2$ by this action 
is a finite volume orbifold with one cusp.  There exists 
a ${\rm PSL}(2,\mathbb{Z})$-invariant collection ${\cal H}$ of open \emph{horoballs} with pairwise disjoint
closure which are
centered at the rational numbers 
and $\infty$ in $\partial \mathbb{H}^2=\mathbb{R}\cup \infty$ (here we use the upper half-plane model
for $\mathbb{H}^2$ and the natural identification of its Gromov boundary $\partial \mathbb{H}^2$ with
$\mathbb{R}\cup \infty$). 
This system of
horoballs is precisely invariant under the action of the group 
${\rm PSL}(2,\mathbb{Z})$: if $H\in {\cal H} $ is such a horoball, and if $g\in {\rm PSL}(2,\mathbb{Z})$
is such that $gH\cap H\not=\emptyset$, then $gH=H$. Furthermore, the action of 
${\rm PSL}(2,\mathbb{Z})$ on ${\cal H}$ is transitive. 
Up to adjusting the system ${\cal H}$,
the complement $X=\mathbb{H}^2-{\cal H}$ 
is a path connected two-dimensional space on which ${\rm PSL}(2,\mathbb{Z})$ 
acts properly and cocompactly. 

Let ${\rm Stab}(H)\subset {\rm PSL}(2,\mathbb{Z})$ be the stabilizer of a component 
$H\in {\cal H}$. Then ${\rm Stab}(H)$ is virtually infinite cyclic, and the hyperbolic group 
${\rm PSL}(2,\mathbb{Z})$ is hyperbolic relative to its system of pairwise conjugate parabolic subgroups
${\rm Stab}(H)$ $(H\in {\cal H})$.
Up to quasi-isometry, 
the Farey graph is then obtained by adding for each $H\in {\cal H}$ a point 
to the Cayley graph of ${\rm PSL}(2,\mathbb{Z})$ and connecting this point to each element in 
${\rm Stab}(H)$ by an edge of length one. Thus a (uniform) quasi-geodesic in ${\rm PSL}(2,\mathbb{Z})$ projects to a 
uniform quasi-geodesic in the Farey graph if and only the length of any subsegment 
which is contained 
in a uniform neighborhood of ${\rm Stab}(H)$ for some $H\in {\cal H}$ is uniformly bounded.

View $\mathbb{H}^2$ as the Teichm\"uller space of marked punctured tori equipped with 
a finite volume hyperbolic metric. Then $X\subset \mathbb{H}^2$ parameterizes such marked tori whose
\emph{systole}, that is, the length of a shortest closed geodesic, is bounded from below
by universal positive constant $\epsilon >0$.

Choose a 
basepoint $x\in X$, and a rose $R\subset x$ with two petals such that the inclusion 
$R\to x$ is an isomorphism on $\pi_1$ and that the $x$-length of $R$ is uniformly bounded.
Since the systole of $x\in X$ is at least $\epsilon$, 
such a rose exists, and it is essentially unique: If $a_1,a_2$ is a free basis of $F_2$ defined
by the petals of the rose, then any other such free  
basis of $F_2$ 
can be obtained from $a_1,a_2$ by a uniformly bounded number of 
Nielsen moves.

Let $\gamma:[0,u]\to {\rm PSL}(2,\mathbb{Z})$ be a uniform quasi-geodesic
through $\gamma(0)={\rm Id}$. Then 
$\gamma$ projects to a uniform quasi-geodesic in the 
Farey graph if and only if for a number $m>0$ depending on the control constants
for the quasi-geodesic, 
the geodesic in $\mathbb{H}^2$ connecting
$x$ to $\gamma(u)(x)$ does not contain any segment of length at least $m$ which is 
contained in $\mathbb{H}^2-X$. Note that this makes sense since each horoball
$H\in {\cal H}$ is convex.

We are left with showing that 
this property is equivalent to $(a,k)$-expo\-nen\-tial growth for some 
$a,k>0$. To this end put $\psi=\gamma(u)$ and consider the unit speed
Teichm\"uller  geodesic segment
$\eta:[0,\tau]\to \mathbb{H}^2$ connecting $x$ to $\psi(x)$, which is 
just the unit speed hyperbolic geodesic. Its length $\tau$ is given as follows. 

Extend $\eta$ to a Teichm\"uller geodesic line, again denoted by $\eta$. 
Its endpoints 
$\eta_+,\eta_-$ 
in $\partial \mathbb{H}^2$ can be thought of as measured geodesic laminations
on the once punctured torus $T$. For $t\in [0,\tau]$ let $q(t)$ be the 
area one singular euclidean
metric on $T$ defined by the area one quadratic differential
which is the cotangent vector of $\eta$ at $\eta(t)$. 
The \emph{length} of $\eta_+$  with respect to $q(t)$ 
contracts along the geodesic
with the contraction rate $e^{-u/2}$, and the length of $\eta_-$ expands with the rate 
$e^{u/2}$. 

For points in $X$,  the singular
euclidean metric defined by an area one quadratic differential is uniformly 
bi-Lipschitz equivalent to the hyperbolic metric in the complement of 
the cusp. The singular euclidean 
length of a simple closed curve $\alpha$ on $x$ (that is, the length of 
a geodesic representative) equals $\iota(\alpha,\eta_+)+\iota(\alpha,\eta_-)$ where
$\iota$ is the \emph{intersection form} on measured lamination space. 
Thus for any subsegment of $\eta$ of hyperbolic length $\kappa$ and with endpoints in 
$X$, the flat length of 
the basis elements $a_1,a_2$ of $F_2$ 
with respect to the hyperbolic metric has increased by at most the factor $e^{\kappa/2}$. 
As for points $z\in X$, this flat length is 
uniformly proportional to the length of the corresponding word with respect to 
a free basis determined by a rose of uniformly bounded length in $z$, 
 $(a,k)$-exponential growth of the path $\gamma$ implies the following.
 There exists a number $c>0$ such that for any $0\leq a< b\leq u$ the length of 
 the hyperbolic geodesic connecting $\gamma(a)(x)$ to $\gamma(b)(x)$ is at least 
 $c(b-a)$.

Now let $\zeta:[0,p]\to \mathbb{H}^2$ be a geodesic arc of length $p>0$
connecting two points on 
the boundary of a horoball $H\in {\cal H}$. Since close-by points in $X$ define
hyperbolic tori which are marked uniformly bi-Lipschitz, we may assume that 
the endpoints of  $\zeta$ are contained in  
the same ${\rm PSL}(2,\mathbb{Z})$-orbit.
This means that there exists an element $\sigma\in {\rm Stab}(H)$ with 
$\sigma(\zeta(0))=\zeta(p)$. Let $\ell >0$ be the word norm of 
$\sigma$ in the infinite cyclic group ${\rm Stab}(H)$. 
Note that this word norm is uniformly proportional to the word norm in 
${\rm PSL}(2,\mathbb{Z})$. Then the 
length $p$ of $\zeta$ is bounded from above by
$b\log \ell+b$ where $b>0$ is a universal constant. As a consequence,
for large enough $\ell$ the condition of $(a,k)$-exponential growth
is violated. In other words, $(a,k)$-exponential growth implies 
property (2) stated in the proposition, which is what we wanted to show.
\end{proof}

We next give an example which shows that for $n\geq 3$, a surgery sequence
which violates the exponential growth condition 
in Theorem \ref{quasigeo2} does not define  in general a uniform 
quasi-geodesic in ${\rm Out}(F_n)$. We use the following preparation.

\begin{lemma}\label{fullcontractsinter}
Let $\Sigma_0,\Sigma$ be simple sphere systems 
and let $\Sigma_i$ be a full surgery sequence of 
$\Sigma_0$ towards $\Sigma$. Let $R\subset M$ be an embedded
rose with $m\leq n$ petals 
such that the inclusion $R\to M$ defines an injection on $\pi_1$ and that
$\iota(R,\Sigma)=m$. 
Then for each $i$ we have
$\iota(\Sigma_i,R)\leq \iota(\Sigma_0,R)+2i$.
\end{lemma}
\begin{proof}
Put the rose $R$ in minimal position with respect to $\Sigma$. This can be 
achieved in such a way that it intersects any component of $\Sigma$ in at most one
point. Let $S$ be a component of $\Sigma$ and let 
$D\subset S$ be an innermost disk for $S- \Sigma_i$ used in the 
surgery which transforms $\Sigma_i$ to $\Sigma_{i+1}$. Assume that the boundary
of $D$ is contained in the component $S_i$ of $\Sigma_i$. 
The disk $D$ has at most one
intersection point with $R$. As a consequence, the two spheres arising from surgery of 
$S_i$ with the disk $D$ intersect $R$ in at most 
$\iota(R, S_i)+2$ points. As the intersection of $R$ with the components of 
$\Sigma_i-S_i$ remains unchanged, a simple induction on 
the length of the surgery sequence yields the lemma.
\end{proof}

We use the lemma to find for any $n\geq 3$  surgery sequences in ${\cal S\cal S\cal G}_n$
which are not quasi-geodesics for an arbitrarily a priori chosen control constant. 
For simplicity of exposition, we only carry out the case 
$n=3$. It will be clear from the discussion that the construction is valid for any 
$n\geq 3$.

\begin{example}\label{noquasigeo}
Consider the free group $F_3$ with a free basis 
${\cal A}_0=\{a_1,a_2,a_3\}$. Let 
$R\subset M$ be a marked rose whose petals define these generators and let 
$\Sigma_0$ be the corresponding dual simple sphere system in $M=\sharp_3 S^1\times S^2$. 
Denote by $S_1\in \Sigma_0$ the sphere which intersects $a_1$. 

Choose a hyperbolic element $\alpha\in {\rm GL}(2,\mathbb{Z})={\rm Out}(F_2)$
and extend it to an element of ${\rm Out}(F_3)$ which preserves $a_1$ (up to a global conjugation).
Denote the thus defined element of 
${\rm Out}(F_3)$ again by $\alpha$. It preserves the conjugacy class of the one-edge free splitting 
$F_3=\langle a_1\rangle * F_2$ where $F_2\subset F_3$ is the free factor generated
by $a_2,a_3$. 
The element $\alpha$ 
acts on the sphere system graph, preserving the sphere $S_1$.
As the element $\alpha$ of ${\rm GL}(2,\mathbb{Z})$ is hyperbolic, it is of exponentially growth.
This is well known but also follows from the proof of Proposition \ref{casen=2}.
As the consequence, the intersection $\iota(\Sigma_0,\alpha^k(R))$ is uniformly exponentially
growing in $k$: there exists a number $c>0$ such that $\iota(\Sigma_0,\alpha^k(R))\geq e^{ck}$.
Furthermore, if $p_2$ denotes the petal of $R$ corresponding to 
the generator $a_2$, then 
the intersection of $\alpha^k(p_2)$ with $\Sigma_0$ is also uniformly exponentially growing.

For each $k$ consider the free basis 
${\cal A}_k=\{a_1\alpha^k(a_2), a_2,a_3\}$ of $F_3$. There exists 
a number $m>0$ and a path in ${\rm Out}(F_3)$ 
of length $2km+1$ which transforms the basis ${\cal A}_0$ to ${\cal A}_k$.
This path consists in first applying $k$ times 
the automorphism $\alpha$, which contributes $km$ to the length of the path. 
The image of ${\cal A}_0$ by this automorphism is the basis 
$a_1,\alpha^k(a_2),\alpha^k(a_3)$ (up to a global conjugation). 
Perform a Nielsen twist to replace 
$a_1$ by $a_1\alpha^k(a_2)$ and iterate $\alpha^{-1}$, extended to $F_3$ by fixing
the free splitting $F_3=\langle a_1\alpha^k(a_2)\rangle *F_2$.
The thus defined path has length $2km+1$, 
and its endpoint $\psi_k$ maps ${\cal A}_0$ to ${\cal A}_k$. Furthermore, 
we have that $\iota(\Sigma_0,\psi_k(R))$ equals $\iota(\Sigma_0,\alpha(p_2))$
up to a universal additive constant and hence these intersection numbers 
are growing exponentially in $k$.

Put $\Lambda_k=\psi_k(\Sigma_0)$. 
Consider the rose $\hat R$ with two petals, obtained from 
the rose $R$ by deleting the petal defining $a_1$. Then 
$\iota(\Sigma_0,\hat R)=\iota(\Lambda_k,\hat R)=2$ for all $k$.
Thus by Lemma \ref{fullcontractsinter}, if $\Sigma_i^k$ is a full surgery sequence connecting
$\Sigma_0$ to $\Lambda_k$, then for each $i$ we have
$\iota(\Sigma^k_i,\hat R)\leq 2i$.  By induction, this implies that 
$\iota(\Sigma^k_i,R)\leq (pi)^2$ for a universal constant $p>0$ and all $i$. Thus, surgery sequences from $\Sigma_0$ to $\Lambda_k$ 
have length growing exponentially in $k$.
As a consequence, the surgery paths do not define a family of uniform quasi-geodesics in 
${\rm Out}(F_3)$.  
\end{example}

\section{Submanifold projection}\label{subfactor}

Let $\sigma_0$ be a nonseparating sphere in
$M=M_n = \#_nS^1\times S^2$. The metric completion $\hat N$ of $M_n-\sigma_0$ with
respect to some path metric on $M_n$
is a compact manifold with two boundary components,
corresponding to the two sides of $\sigma_0$. The manifold $N$ 
obtained by gluing a $3$-ball to each boundary component of
$\hat N$  is homeomorphic to $M_{n-1}$.
Our goal is to analyze intersections of spheres with 
$\hat N$ and use this to define a \emph{submanifold projection} 
of the sphere graph of $M$ into the sphere graph of $N$.

We begin with a topological observation. 

\begin{lemma}\label{lem:kappen-uberleben}
  Let $S$ be any sphere in normal position with respect to
  $\sigma_0$ which is not disjoint from $\sigma_0$.
  Let $D\subset S$ be any innermost disk of 
  $S-\sigma_0$, and let
  $D_0\subset\sigma_0$ be an embedded disk in $\sigma_0$ with the same
  boundary circle: $\partial D = \partial D_0$. Then the sphere
  $S^\prime = D\cup D_0$ is essential in $N$.
\end{lemma}
\begin{proof}
  Assume by contradiction that $S'$ is inessential in $N$.
  Denote the boundary component of $\hat N$ which intersects the disk $D$
  by $\partial^+\hat N$ and the other by $\partial^-\hat N$.
  Equip $\sigma_0$ with the orientation of the oriented
  boundary component $\partial^+\hat N$ 
  of $\hat N$ (for a choice of an orientation of $N$). Since $\sigma_0$ is 
  non-separating by assumption,
  this choice of orientation determines a
  choice of a generator of $H_2(M,\mathbb{Z})$, given by the oriented 
  inclusion $\sigma_0\to M$, again denoted by
  $\sigma_0$. Furthermore, this choice of orientation 
  restricts to an orientation of $D_0$ and hence defines an 
  orientation of $S^\prime$.

  Since $S'$ is
  an inessential embedded sphere in $N$, 
  it bounds a ball in $N$. Because $\sigma_0$ and $S$ are in minimal
  position, the sphere $S'$ does not bound a ball in the manifold
  $\hat N$.
  Similarly, the sphere $S^\prime$ does not bound a
  ball in the manifold $\hat N_+$
  obtained from $\hat N$ by gluing a ball to $\partial^+\hat N$. 
  Namely, otherwise $D$ would be homotopic in $\hat N$
  into $\partial^+\hat  N$,
  violating as before normal position. 
  As a consequence, $S^\prime$ bounds a region in $\hat N$ whose
  second boundary component is $\partial^-\hat N$.
  Thus $S^\prime$ is homologous to $\pm \sigma_0$ in $M$.
  Inspecting orientations,
  we obtain that $S^\prime$ defines the homology class $\sigma_0$
  in $M$.

  Let $\hat S$ be the sphere in $N$ obtained by gluing $\sigma_0-D_0$ to 
  $D$ and equipped with the orientation inherited from the
  boundary orientation of $\partial^+\hat N$. For this choice of orientation,  
  $\sigma_0$ is the oriented connected sum of $S^\prime$ and 
  $\hat  S$. Thus as homolopy classes in $M$, we have
  $\sigma_0=S^\prime+\hat S=\sigma_0+\hat S$
  and hence 
  $\hat S$ is homologically trivial in $M$. In other words, the
  embedded sphere $\hat S$ in $M$ is separating. Furthermore, it is not
  homotopically trivial in $M$, again by minimal position.

  Now the second homotopy group $\pi_2(M)$ of $M$ is a free
  $\pi_1(M)$-module which is the direct sum of two submodules
  $V_1\oplus V_2$, where $V_1$ is spanned by nonseparating embedded spheres
  and $V_2$ is spanned by separating embedded spheres. In other words,
  $V_2$ is the kernel of the map $\pi_2(M)\to H_2(M,\mathbb{Z})$ as $\pi_1(M)$-modules, 
  where the action of $\pi_1(M)$ on $H_2(M,\mathbb{Z})$ is the trivial action.
  By the above, the spheres $\sigma_0$ and $S^\prime$ are contained in the submodule 
  $V_1$, and the sphere $\hat S$ is contained
  in $V_2$. As $\sigma_0+S^\prime=\hat S$ (connected sum and hence
  sum in $\pi_2(M)$), and all elements are non-zero, 
  this is impossible.
\end{proof}

  %
%
  %
  %
  %
  %

Call a sphere $S\subset M-\sigma_0$ \emph{non-peripheral} if its image 
in the manifold $N$ 
is non-trivial. The set of all non-peripheral spheres defines a subgraph 
${\cal N\cal P}(\sigma_0)$ of the sphere graph of $M$ consisting of sphere 
disjoint from $\sigma_0$. Let also ${\cal P}(\sigma_0)$ be the 
set of all spheres which 
are disjoint from $\sigma_0$ and peripheral. Note that any such sphere 
(with $\sigma_0$ excluded) is separating.

Lemma~\ref{lem:kappen-uberleben} allows to define a \emph{submanifold
  projection}
\[p_{\sigma_0}:{\cal S\cal G}_n-{\cal P}(\sigma_0) \to {\cal N\cal P}(\sigma_0)\]
(more precisely, the target of the 
projection is the family of all non-empty finite subsets of ${\cal N\cal P}(\sigma_0)$)
in the
following way. For a sphere 
$S$ in $M$ distinct from $\sigma_0$ and not peripheral we put 
$p_{\sigma_0}(S)=S$ if $S$ is disjoint from $\sigma_0$, and 
if $S$ intersects
$\sigma_0$, then we let $p_{\sigma_0}(S)$ be the union of
all spheres which 
are 
obtained by surgery at an innermost disk of $S-\sigma_0$. By
Lemma~\ref{lem:kappen-uberleben}, each such surgery yields 
a non-peripheral sphere in $M-\sigma_0$. 
The projection $p_{\sigma_0}(\Sigma)$ of a sphere system $\Sigma$
with more than one component is defined to be the union
$\cup_{S\in \Sigma}p_{\sigma_0}(S)$.

There may be spheres in the set $p_{\sigma_0}(\Sigma)$ which intersect,
but as a subset of ${\cal N\cal P}(\sigma_0)$ it is of uniformly bounded 
diameter. 
Namely, all innermost disks of
$\Sigma-\sigma_0$ are disjoint. 
Hence if $S_1,S_2\in p_{\sigma_0}(\Sigma)$
are any two spheres constructed from innermost disks $D_1,D_2$, 
then there exist disjoint spheres $S_1^\prime,S_2^\prime\in 
p_{\sigma_0}(\Sigma)$ 
such that $S_i^\prime$ is disjoint from $S_i$
$(i=1,2)$. Just 
choose $S_i^\prime$ to be
the spheres constructed from two innermost disks $D_1,D_2$ of $\Sigma- \sigma_0$ and 
two disjoint disks in $\sigma_0$ bounded by the disjoint boundary circles of 
$D_1,D_2$. 

Let ${\cal S\cal G}_{N}={\cal S\cal G}_{n-1}$ be the sphere graph of the 
manifold $N$ obtained by cutting $M$ open along $\sigma_0$ and capping off the boundary.
There exists a natural simplicial projection 
\[\Upsilon_{\sigma_0}:{\cal N\cal P}(\sigma_0)\to 
{\cal S\cal G}_{N}.\] Consider the composition 
\[p_N=\Upsilon_{\sigma_0}\circ p_{\sigma_0}:{\cal S\cal G}_n-{\cal P}(\sigma_0)\to 
{\cal S\cal G}_{N}.\]  

Our next goal is to establish a control of the images of suitably chosen surgery paths under
the projections $p_N$. 
This will follow from 
a stability property of normal position along such surgery sequences.
Note that the disk $D_2$ in the formulation of the lemma below
need not be innermost, which corresponds to the 
second possibility listed.

\begin{lemma}\label{lem:stability-kappen}
  Let $\Sigma_1$ be a sphere system, and let $\Sigma_2$ be a sphere
  which is in normal
  position with respect to $\Sigma_1$. Let $D_i \subset \Sigma_i, i=1,2$ be two embedded disks such
  that $\partial D_1 = \partial D_2$, and such that the interiors of
  the $D_i$ are disjoint. Let $S= D_1\cup D_2$. Then either
  \begin{enumerate}
  \item up to homotopy, $S$ is disjoint from $\Sigma_1$, or
  \item the normal position of $S$ with respect to $\Sigma_1$ has an
    innermost disk component which is (with boundary gliding on $\Sigma_1$)
    isotopic to an innermost disk component of $\Sigma_2$.
  \end{enumerate}
\end{lemma}

\begin{proof}
  Let $\widetilde{M}$ be the universal cover of $M$. We let
  $\widetilde{\Sigma}_1$ be the full preimage of $\Sigma_1$, and let
  $\overline{\Sigma}_2$ be a connected lift of $\Sigma_2$. This contains a
  unique lift $\overline{D}_2$ of $D_2$. We denote by $\overline{D}_1$
  the unique lift of $D_1$ which intersects $\overline{D}_2$. Then the
  sphere
  $$\overline{S}= \overline{D}_1 \cup \overline{D}_2$$
  is a connected lift of $S$. We modify $S$ and this lift by pushing
  $\overline{D}_1$ slightly off $\widetilde{\Sigma}_1$ in order to make
  every intersection of $\overline{S}$ with $\widetilde{\Sigma}_1$
  transverse. Note that every such intersection circle is then
  contained in $\overline{D}_2$. In particular, every innermost disk
  component of $\overline{S}$ with respect to $\widetilde{\Sigma}_1$ is
  either contained in $\overline{\Sigma}_2$, or contains $\overline{D}_1$
  (and there is at most one of the latter type).

  If there is no innermost disk containing 
  $\overline{D}_1$, or if the innermost disk containing 
  $\overline{D}_1$ is not homotopic (relative to its boundary) into
  $\widetilde{\Sigma}_1$, then $S$ is in normal position with respect to
  $\Sigma_1$. Namely, any other pathology is excluded by normal
  position of $\Sigma_1$ and $\Sigma_2$. In that case we find an
  innermost disk component of $\overline{S}$ which is contained in 
  $\overline{D}_2$ and satisfies 
  property ii).
   
  If there is an innermost disk component $D\supset \overline{D}_1$
  which is homotopic relative to its boundary into $\widetilde{\Sigma}_1$,
  then there is a ball $B$ whose boundary is the union of $D$ with a
  disk contained in $\widetilde{\Sigma}_1$. In this case, we can homotope
  $S$ by pushing it through this ball. As a result we obtain a sphere
  $S'$ which is again of the form $D'_1 \cup D'_2$ with disks
  contained in $\Sigma_1,\Sigma_2$, and whose lift intersects $\widetilde{\Sigma}_1$
  in one less circle. Iterating this argument, we either terminate in
  a sphere which is disjoint from $\widetilde{\Sigma}_1$ and therefore 
  has property i), or we obtain property ii) as above.
\end{proof}

Lemma \ref{lem:stability-kappen} allows to define \emph{nested} surgery sequences
$\Sigma_i$ as follows. Let $\Sigma_0,\Sigma$ be sphere systems, and let 
$S_0,S$ be components of $\Sigma_0,\Sigma$ which intersect. Choose an innermost
disk $D\subset S\in \Sigma$ with boundary on $S_0\in \Sigma_0$ and let 
$D_0\subset S_0$ be a disk with boundary 
$\partial D_0=\partial D$. Perform surgery of $S_0$ by replacing $S_0$ by
$S_1=D_0\cup D$.

Assume that $S_1$ is not disjoint from $S$. Choose an innermost disk
$D^\prime$ of $S-S_1$, with boundary on $S_1$. By Lemma \ref{lem:stability-kappen},
up to homotopy, the boundary of $D^\prime$ is contained in $D_0$ and hence
bounds a unique disk $D_1\subset D_0$. Perform surgery by replacing
$S_1$ by $D_1\cup D^\prime$ and iterate this construction.

\begin{lemma}\label{nested}
Let $(S_i)$ be a nested surgery sequence of a sphere $S_0$ towards a sphere
$S$. Then each sphere $S_i$ in the sequence is a union of a disk $D_i\subset S_0$
and a disk $D_i^\prime\subset S$, with $D_{i+1}\subset D_{i}$.
\end{lemma}
\begin{proof} We proceed by induction on the length $m$ of the sequence. 
The statement is clear in the case $m=1$, so assume that the statement holds true for 
$m-1$. Let $(S_i)$ be a nested surgery sequence of length $m$. By induction 
hypothesis, $S_{m-1}$ is a union of a disk $D_{m-1}\subset S_0$ and a disk
$D_{m-1}^\prime\subset S$. By Lemma \ref{lem:stability-kappen}, an innermost disk 
$D\subset S$ of $S-S_{m-1}$ has its boundary in $D_{m-1}$ and hence
bounds a disk $D_m\subset D_{m-1}$. Moreover up to homotopy, 
either $D$ contains the disk $D^\prime_{m-1}$ and hence the sphere $S_m$ obtained from $S_{m-1}$ by
nested surgery with innermost component $D$ is a union of 
$D\supset D_{m-1}^\prime$ and $D_m\subset D_{m-1}$, or it is disjoint from $D_{m-1}^\prime$ and 
once again, the statement of the lemma is true for $S_m$. 
\end{proof}

Let as before $d_{\cal S\cal G}$ be the distance in the sphere
graph of $M$.

\begin{lemma}\label{lem:projection-doesnt-change-after-2}
  Let $(S_i)$ be a nested surgery sequence connecting 
  a non-separating sphere $\sigma_0$
  to a different sphere $S$. Let $S_k$ be any point
  of this surgery sequence which satisfies 
  $d_{\cal S\cal G}(S_k, \sigma_0) \geq 2$. Let $\hat N$
  be obtained from $M$ by
  removal of $\sigma_0$, and let $N$ be obtained from $\hat N$ by 
  capping off the boundary spheres. Then 
    \[p_{N}(S_k)\cap p_{N}(S)\not=\emptyset.\]
   Consequently, the projections $p_N(S_k),p_N(S)$ are $2$-close
   in the sphere graph of $N$.
\end{lemma}
\begin{proof}
  Assume without loss of generality that we have chosen
  representatives of $\sigma_0,S$ which are in normal
  position. We will denote these representatives by the same symbol
  again.

  Let $S_i$ be a sphere on the nested surgery sequence. 
  By Lemma \ref{nested}, $S_i$ is a union of
  two disks $S_i=D_i^- \cup D_i^+$ with $D_i^-\subset \sigma_0$ and $D_i^+
  \subset S$.

  By Lemma~\ref{lem:stability-kappen}, either $S_i$ is disjoint
  from $\sigma_0$, or its normal position has an innermost disk
  component which is also an innermost disk component of
  $S$. In the latter case, the projections $p_{N}(S_i),p_{N}(S)$ intersect
  as stated in the lemma.

  The final statement of the lemma follows from Lemma~\ref{lem:kappen-uberleben}.
\end{proof}

Now we can show the \emph{bounded geodesic projection theorem} using an
argument of Webb from \cite{We15}.

\begin{theo}\label{thm:bounded-geodesic-projection}
  There is a number $q>0$ with the following property.
  Let $\sigma_0$ be a nonseperating sphere, and let 
  $N$ the capped off 
  complement of $\sigma_0$ as before, with innermost 
  projection $p_N:{\cal S\cal G}_n-{\cal P}(\sigma_0)\to {\cal S \cal G}_{N}$.
  
  Let $(S_i)_{0\leq i\leq m}$ be any geodesic in ${\cal S\cal G}_n$ which is 
  disjoint from $\mathcal{P}(\sigma_0) \cup \{\sigma_0\}$. Then
  $$d(p_{N}(S_0), p_{N}(S_m)) < q.$$
\end{theo}
\begin{proof} By Theorem 1.2 of \cite{HiHo17}, there exists a number $K>0$ such that 
  surgery sequences are unparameterized
  $K$-quasigeodesics in the sphere graph. Since the sphere graph is
  Gromov hyperbolic, there is a constant $D>0$ such that a triangle with
  $K$-quasigeodesic sides is $D$-thin.

  For ease of exposition, we distinguish between two cases. First
  assume that the geodesic $(S_i)$ never enters the
  $(2D+2)$-neighborhood of $\sigma_0$.
  Consider nested surgery sequences $P$ and $Q$ joining $\sigma_0$ to 
  spheres disjoint from 
  $S_0$ and $S_m$, respectively.

  By the thin triangle property, there is a sphere $S_k$ which is of distance at most $D$ to
  both $P$ and $Q$. Furthermore, every point $S_i$ for $i<k$ is of distance at
  most $D$ to $P$, and every point $S_i$ for $i>k$ is of distance at most $D$
  to $Q$. 
  By Lemma~\ref{lem:projection-doesnt-change-after-2}, the projections
  to $N$ of any point on $P$ of distance at least $2$ from $\sigma_0$ 
  intersect and hence are
  coarsely the same. Since $d_{\cal S\cal G}(S_i,\sigma_0)\geq 2D+2$ for all $i$, 
  for $i\leq k$ the sphere $S_i$ is of distance at most $D$ from a point $S_i^\prime$ on 
  $P$ of distance at least $D+2$ from $\sigma_0$. Thus a geodesic connecting 
  $S_i^\prime$ to $S_i$ does not enter the $1$-neighborhood of $\sigma_0$. Hence the projection
  $p_N$ is defined on such a geodesic, and since $p_N$-is $2$-Lipschitz, 
  the projection of $S_i,
  i\leq k$ is coarsely equal to the projection of $S_k$, and
  similarly for $S_i, i\geq k$. This shows that the 
  diameter of the projection is bounded from above by a universal constant as claimed.

  If $(S_i)$ does enter the $(2D+2)$-ball around $\sigma_0$, then the
  argument needs to be modified in the following way. Let $(S_i)_{j\leq i\leq u}$ be the
  minimal connected segment in $(S_i)$ which contains all intersection
  points of $(S_i)$ with the $(2D+2)$ -ball around $\sigma_0$. 
  The diameter of this segment is at most $4(D+1)$, and since $(S_i)$ is a geodesic, the same is true for its length.
  By our assumption that $(S_i)$ is disjoint from $\mathcal{P}(\sigma_0)\cup\{\sigma_0\}$, the projection 
  $p_N(S_i)$ is defined for all $i$, and the assignment $i \mapsto p_N(S_i)$ is $2$--Lipschitz.
  Hence, we have
  \[ \mathrm{diam}(\{ p_N(S_i), \quad j \leq i \leq u \}) \leq 8(D+1). \]
  On the complement of $(S_i)_{j\leq i\leq u}$ the argument used in the first case applies. Together
  this completes the proof. 
\end{proof}



The condition in the theorem simplifies for nonseparating spheres. 
To exploit this, recall from Lemma \ref{nonseparating} that any two 
non-separating spheres in ${\cal S\cal G}_n$ can be connected by a geodesic
consisting of non-separating spheres. We use this in the following

\begin{corollary}\label{bigprojection}
    Suppose that $\sigma_1, \sigma_2$ are two non-separating spheres, and 
    that $(S_i)$ is a geodesic in the sphere graph
    connecting $\sigma_1$ to $\sigma_2$ consisting of non-separating spheres. 
    
    If $N$ is the capped off complement of a non-separating
    sphere $\sigma_0$ (as in Theorem~\ref{thm:bounded-geodesic-projection}), and
    \[ d(p_N(\sigma_1), p_N(\sigma_2))  \geq q, \]
    then $\sigma_0 = S_i$ for some $i$.
\end{corollary}
\begin{proof}
    As remarked above, any sphere in $\mathcal{P}(\sigma_0)$ distinct from $\sigma_0$
    is separating. Hence, if $\sigma_0 \neq S_i$ for all $i$, then 
    the geodesic 
    $(S_i)$ consisting of non-separating spheres 
    satisfies the assumption in Theorem~\ref{thm:bounded-geodesic-projection}. This yields the desired contradiction.
\end{proof}

A useful more general version of this corollary is the following

\begin{corollary}\label{quasigeobounded}
For every $L>0$ there exists a number $q(L)>0$ with the following property. Let 
$(S_i)_{0\leq i\leq m}$ be an $L$-quasi-geodesic edge path in the graph of non-separating spheres. If $N$ is the capped of complement
of a non-separating sphere $\sigma_0$ and if $d(p_N(S_0),p_N(S_m)) \geq q(L)$
then $\sigma_0=S_i$ for some $i$.
\end{corollary}
\begin{proof}
We know that a uniform quasi-geodesic in ${\cal S\cal G}_n$ avoiding 
${\cal P}(\sigma_0)$ 
has uniformly small diameter projection into ${\cal S\cal G}_N$. Thus if the diameter of the 
projection is large, it has to pass through ${\cal P}(\sigma_0)$. As any 
point in ${\cal P}(\sigma_0)$ is separating, if the path consists of non-separating spheres then
it has to pass through 
$\sigma_0$.
\end{proof}

\section{Actions of ${\rm Out}(F_n)$ on products of hyperbolic spaces}\label{sec:products}

This final section is devoted to the proofs of the 
results stated in the introduction.
We follow the strategy developed in \cite{BBF15} as used in \cite{BF14b}. The starting point is 
the following result of \cite{BBF15}.

\begin{theo}\label{BBF}
Let ${\cal Y}$ be a collection of $\delta$-hyperbolic spaces, and for every pair
$A,B\in {\cal Y}$ of distinct elements suppose that we are given a uniformly bounded
subset $\pi_A(B)\subset A$, called the projection of $B$ to $A$. Denoting by
$d_A(B,C)$ the diameter of $\pi_A(B)\cup \pi_A(C)$, assume that the following holds: there is
a constant $K>0$ such that 
\begin{enumerate}
\item if $A,B,C\in {\cal Y}$ are distinct, then at most one of the three numbers
\[d_A(B,C), \quad d_B(A,C),\quad d_C(A,B)\]
is greater than $K$ and 
\item for any distinct $A,B$ the set 
\[\{C\in {\cal Y}-\{A,B\}\mid d_C(A,B)>K\}\]
is finite.
\end{enumerate}
Then there is a hyperbolic space $Y$ and an isometric embedding of each $A\in {\cal Y}$ 
onto a convex set in $Y$ so that the images are pairwise disjoint and the nearest point
projection of any $B$ to any $A\not=B$ is within uniformly bounded distance
of $\pi_A(B)$. Moreover, the construction is equivariant with respect to any group acting 
on ${\cal Y}$ by isometries.
\end{theo}

For a non-separating sphere $S\subset M$ let ${\cal Y}(S)$ be the 
following graph. The set of vertices of ${\cal Y}(S)$ 
is the set ${\cal N\cal P}(S)$ of non-peripheral spheres in 
$M-S$. Two such spheres are connected by an edge of length one if 
their projections into the sphere graph ${\cal S\cal G}_S$ of the manifold 
obtained from $M-S$ by capping off the boundary are of distance at most one.
With this definition, the graph ${\cal Y}(S)$ is a geodesic metric graph which 
is $2$-quasi-isometric to the graph ${\cal S\cal G}_S$ and hence it is 
$\delta$-hyperbolic for a constant $\delta >0$ not depending on $S$.
The group ${\rm Out}(F_n)$ acts on the collection 
${\cal Y}=\{{\cal Y}(S)\mid S\}$ by isometries. 

For  
$S$ let $p_S:{\cal S\cal G}_n-{\cal P}(S)\to {\cal Y}(S)$ be the submanifold projection defined in 
Section \ref{subfactor}. Note that in contrast to the construction in Section \ref{subfactor}, the 
target of the map $p_S$ equals the set ${\cal N\cal P}(S)$ equipped with a metric inherited from 
${\cal S\cal G}_S$. Thus it makes sense to project the image into the complement of other spheres which may
intersect $S$. 
If $A$ is a non-separating sphere and if 
$S$ is contained in $M-A$, then $p_S$ is not defined on all of 
${\cal N\cal P}(A)={\cal Y}(A)$, but the only 
exceptions are points in ${\cal P}(S)$. 
Extend the definition of $p_S$ to ${\cal Y}(A)$ by putting
\[p_S({\cal P}(S)\cap {\cal N\cal P}(A)) =p_S(A).\] Note that this 
should be viewed as an extension of 
$p_S$ to all of ${\cal Y}(A)$. This extension depends on $A$, but the collection of these 
extension is equivariant with respect to the action of ${\rm Out}(F_n)$.

\begin{proposition}\label{projection}
The collection $({\cal Y}(S),p_S)$ satisfies the conditions in Theorem \ref{BBF}.
\end{proposition}
\begin{proof}
Let $B$ be a non-separating sphere different from $S$. 
We begin with showing that the diameter of the set
$p_S({\cal Y}(B))\subset {\cal Y}(S)$ is uniformly bounded, independent of $B$ and $S$.

To this end we distinguish two cases. In the first case we have
$d_{\cal S\cal G}(B,S)\geq 2$.
Then $B$ intersects $S$, furthermore for every 
$C\in {\cal N\cal P}(B)-{\cal P}(S)$, the projections
$p_S(B),p_S(C)$ contain components which are disjoint and hence 
whose distance in ${\cal Y}(S)$ equal one. 
By the definition of $p_S$, this implies that 
$p_S({\cal N\cal P}(B))$ is contained in a uniformly bounded neighborhood of 
$p_S(B)$.

If $d_{\cal S\cal G}(B,S)=1$ then $B\in {\cal N\cal P}(S)$ since $B$ is non-separating.
Then for any $C\in {\cal N\cal P}(B)-{\cal P}(S)$, the projection $p_S(C)$ is disjoint 
from $B$. Once again, by the definition of the projection $p_S$, we conclude that
$p_S({\cal N\cal P}(B))$ is contained in a uniformly bounded neighborhood of $B$.
This completes the proof  that the diameters of the sets  $p_S({\cal Y}(B))$ $(B\not=S)$
are bounded from above by a constant not depending on $B,S$.

We next verify property (1) in Theorem \ref{BBF}. Thus 
let $A,B,C$ be pairwise distinct non-separating spheres and 
suppose that $d_A(B,C)>2q$ where $q>0$ is as in Theorem \ref{thm:bounded-geodesic-projection}. 
Choose a geodesic $\gamma$ connecting $B$ to $C$ consisting of non-separating spheres. 
By Corollary~\ref{bigprojection}, 
the geodesic $\gamma$ has to pass through $A$. Let $i\geq 1$ be such that
$\gamma(i)=A$. Then Theorem \ref{thm:bounded-geodesic-projection} shows 
that $d_B(C,\gamma(i+1))\leq q$. As $A$ and $\gamma(i+1)$ are disjoint, the distance
between $p_B(A),p_B(\gamma(i+1))$ is uniformly bounded and hence the same holds true
for $d_B(A,C)$.

As the roles of $B,C$ can be
exchanged, this shows that condition (1) in Theorem \ref{BBF} is fulfilled.

Property~(2) follows immediately from Corollary~\ref{bigprojection}: the only spheres $C$ so that the projection $d_C(A,B)$ is large 
appear along a (fixed) geodesic consisting only of nonseparating spheres, which has finite length.
\end{proof}

As a fairly immediate consequence, we obtain a 
more precise version of the main result of \cite{BF14b} in rank 3.

\begin{corollary}\label{rank3}
The group ${\rm Out}(F_3)$ admits an isometric action on a product 
$Y=Y_1\times Y_2$ of two hyperbolic metric spaces so that every exponentially growing 
automorphism has positive translation length.
\end{corollary}
\begin{proof}
By Theorem \ref{BBF} and Proposition \ref{projection}, the group ${\rm Out}(F_3)$ admits an isometric
action on $Y=Y_1\times Y_2$ where $Y_1$ is the free splitting complex or, equivalently, the sphere graph of $M=M_3$,
and where $Y_2$ is a hyperbolic space containing for each non-separating sphere $S$ the graph
${\cal Y}(S)={\cal S\cal G}_2$ as a convex isometrically embedded subspace. 

A non-separating sphere $S$ in the manifold $M$ corresponds precisely to the conjugacy class of 
a corank one free factor, consisting of homotopy classes of loops based at a point
$p\in M-S$ which do not intersect $S$. As a consequence, any element of 
${\rm Out}(F_3)$ which preserves such a corank one free factor, defined 
by the sphere $S$, and acts as an exponentially growing 
automorphism on it acts with positive translation length on the graph ${\cal Y}(S)$, which 
is uniformly quasi-isometric to the Farey graph. Then such an element acts with positive 
translation length on $Y_2$ and hence on $Y$.
We refer to Section \ref{sec:spontaneous}
for a detailed discussion. 

On the other hand, by \cite{HM19}, if $\phi$ is an exponentially growing automorphism of $F_3$
then there exists a number $j\geq 1$ such that either $\phi^j$ acts with positive translation length
on the sphere graph $Y_1$ of $M$, or $\phi^j$ preserves a corank one free factor $A$ and acts with 
positive translation length on the free splitting complex of $A$. Note that 
the conclusion on the corank stems from the fact that a corank 2 free factor of $F_3$ is infinite cyclic and hence
does not admit any exponentially growing automorphisms. Together this yields the proof of the 
corollary.
\end{proof}

The above construction  can be interpreted in the following way.
Let $n\geq 3$ and let ${\cal P\cal G}_n$ be the graph whose vertices are 
ordered pairs $(S_1,S_2)$ of disjoint non-separating spheres. 
Two such pairs 
$(S_1,S_2)$ and $(S_1^\prime,S_2^\prime)$ are connected by an edge of length
one if either $S_1=S_1^\prime$ and the second spheres 
$S_2,S_2^\prime$ are connected
by an edge in the graph ${\cal Y}(S_1)$, or if $(S_1^\prime,S_2^\prime)=(S_2,S_1)$.
Note that in contrast to similar constructions for graphs of curves or graphs of disks
(see for example \cite{H16}), the spheres
$(S_1,S_2)$ and $(S_1^\prime,S_2^\prime)$ may be connected
by an edge although they can not be realized disjointly. The group 
${\rm Out}(F_n)$ acts on the graph 
${\cal P\cal G}_n$ as a group of simplicial automorphisms. 

Since two spheres in the first factor of the points in ${\cal P\cal G}_n$ only are exchanged
if they are disjoint, the first factor projection 
$\Pi_1:{\cal P\cal G}_n\to {\cal S\cal G}_n$ is an 
${\rm Out}(F_n)$-equivariant one-Lipschitz projection onto
the 1-dense convex subgraph of non-separating spheres. Note that ${\cal P\cal G}_n$ is only 
defined for $n\geq 3$. 

\begin{theo}
The graph ${\cal P\cal G}_n$ of non-separating pairs is a hyperbolic 
${\rm Out}(F_n)$-graph.
\end{theo}
\begin{proof}
Given what we achieved so far, the proof is fairly standard. For each non-separating 
sphere $S\subset M$ consider the subgraph 
\[\Pi_1^{-1}(S)=\{(S,S^\prime)\mid S^\prime\}=H(S)\subset {\cal P\cal G}_n\] 
of pairs
with one component equal to $S$. This graph is $2$ -quasi-isometric to ${\cal Y}(S)$ and hence it is 
$\delta$-hyperbolic for a number $\delta >0$ not depending on $S$. 

For $S\not=S^\prime$, the intersection $H(S)\cap H(S^\prime)$ can be viewed as a graph 
of non-separating spheres which are 
disjoint from both $S,S^\prime$. Thus the diameter of this intersection in both 
$H(S),H(S^\prime)$ is uniformly bounded.

Let ${\cal E\cal G}$ be the electrification of ${\cal P\cal G}_n$ with respect to the family ${\cal H}$ of 
subgraphs $H(S)$. This electrification is the graph obtained from ${\cal P\cal G}_n$ by
adding a vertex $v_S$ for each of the graphs $H(S)$ and connecting $v_S$ to each vertex in 
$H(S)$ by an edge. 
By construction, this electrification is two-quasi-isometric to the graph of non-separating 
spheres and hence it is hyperbolic. In particular, any $L$-quasi-geodesic in 
${\cal E\cal G}$ defines a $2L$-quasi-geodesic in ${\cal S\cal G}_n$. 

The \emph{bounded penetration property} in this context states 
that for every $L>1$ there exists a number $p(L)>0$ with the following property
\cite{H16}. Call an $L$-quasi-geodesic edge path in ${\cal E\cal G}$ \emph{efficient} if 
for every non-separating sphere $S$ we have $\gamma(k)=v_S$ for at most one $k$.
Let $\gamma\subset {\cal E\cal G}$ be an efficient $L$-quasi-geodesic and let
$S$ and $k$ be such that $\gamma(k)=v_S$. If the distance in $H(S)$ between 
$\gamma(k-1)$ and $\gamma(k+1)$ is at least $p(L)$ then every efficient 
$L$-quasi-geodesic $\gamma^\prime$ in ${\cal E\cal G}$ with the same endpoints as $\gamma$
passes through $v_S$. Moreover, if $\gamma^\prime(k^\prime)=v_S$ then the distance 
in $H(S)$ between $\gamma(k-1),\gamma^\prime(k^\prime-1)$ and
$\gamma(k+1),\gamma^\prime(k^\prime+1)$ is at most $p(L)$.

By Corollary \ref{quasigeobounded} and the fact that ${\cal E\cal G}$ is 
$2$-quasi-isometric to the graph of non-separating spheres,  
the bounded penetration property holds true for the subspaces $H(S)$. Thus 
it follows from Theorem 1 of \cite{H16} that ${\cal P\cal G}_n$ is hyperbolic.
\end{proof}

The graph ${\cal P\cal G}_n$ also has the following description. Its vertices are 
conjugacy classes of pairs $A_1>A_2$ of free factors, where $A_1$ is of corank $1$ and $A_2$ is of corank $2$.
There are two types of edges. The first type preserves $A_1$ and exchanges $A_2$ by a corank one free factor 
connected to $A_2$ by an edge in the free splitting graph of $A_1$. 
The second type preserves $A_2$ and replaces $A_1$ by a corank one free factor 
containing $A_2$ which is connected to 
$A_1$ by an edge in the free splitting graph of $F_n$.

Note that the group ${\rm Out}(F_n)$ naturally acts on ${\cal P\cal G}_n$ as a group of 
simplicial isometries.
Using this graph we can complete the proof of Theorem \ref{main}.

\begin{theo}\label{rank3sharp}
The group ${\rm Out}(F_3)$ admits an isometric action on a hyperbolic metric graph such that
every exponentially growing automorphism has positive translation length.
\end{theo}
\begin{proof}
The proof is immediate from the proof of Corollary \ref{rank3} via noting that by 
Theorem 1 of \cite{H16} and the construction of the graph ${\cal P\cal G}_n$, for each 
non-separating sphere $S$ the subgraph $H(S)$ is uniformly quasi-convex and isometric
to the graph ${\cal Y}(S)$. Thus any exponentially growing automorphism of $F_3$ acts with 
positive translation length on ${\cal P\cal G}_3$. 
\end{proof}

\noindent
MATHEMATISCHES INSTITUT DER UNIVERSIT\"AT BONN\\
ENDENICHER ALLEE 60, D-53115 BONN\\
e-mail: ursula@math.uni-bonn.de

\noindent
MATHEMATISCHES INSTITUT DER LMU M\"UNCHEN\\
THERESIENSTRASSE 39, D-80333 M\"UNCHEN\\
e.mail: hensel@math.lmu.de


\end{document}